\documentclass{article}
\usepackage{latexsym}
\usepackage{mathrsfs}
\usepackage{indentfirst}
\usepackage{amsmath,amsthm, amssymb}
\usepackage[dvips]{graphicx}
\usepackage{graphicx}
\usepackage{epsfig}

\newtheorem{theorem}{Theorem}[section]
\newtheorem{lemma}[theorem]{Lemma}

\newtheorem{problem}{Problem}

\newtheorem{conjecture}{Conjecture}

\begin{document}
\title{Equivalence between Extendibility and Factor-Criticality \thanks{Work supported by the
Scientific Research Foundation of Guangdong Industry Technical
College, granted No. 2005-11.}}
\author{Zan-Bo Zhang$^1$$^2$\thanks{Corresponding Author. Email address: eltonzhang2001@yahoo.com.cn.}, Tao Wang$^3$, Dingjun
Lou$^1$
\\ $^1$Department of Computer Science, \\ Sun Yat-sen University, Guangzhou 510275, China
\\ $^2$Department of Computer Engineering, Guangdong \\ Industry Technical College, Guangzhou 510300, China
\\ $^3$Center for Combinatorics, LPMC, \\ Nankai University, Tianjin 300071, China
}
\date{}
\maketitle
\thispagestyle{empty}
\begin{abstract}
In this paper, we show that if $k\geq (\nu+2)/4$, where $\nu$
denotes the order of a graph, a non-bipartite graph $G$ is
$k$-extendable if and only if it is $2k$-factor-critical. If
$k\geq (\nu-3)/4$, a graph $G$ is $k\frac{1}{2}$-extendable if and
only if it is $(2k+1)$-factor-critical. We also give examples to
show that the two bounds are best possible. Our results are
answers to a problem posted by Favaron \cite{F1} and Yu \cite{Y1}.
$\newline$\noindent\textbf{Key words}:
$n$-factor-critical, $n$-critical, $k$-extendable,
$k\frac{1}{2}$-extendable
\end{abstract}
\section{Introduction, terminologies and preliminary results}
All graphs considered in this paper are finite, connected,
undirected and simple. Let $G$ be a graph, vertex set and edge set
of $G$ are denoted by $V(G)$ and $E(G)$. Let $S\subseteq V(G)$, we
use $G[S]$ to denote the subgraph of $G$ induced by $S$ and $G-S$
to denote the subgraph $G[V(G)\backslash S]$. Let $G_{1}$ and
$G_{2}$ be two disjoint graphs. The \emph{union} $G_1\cup G_2$ is
the graph with vertex set $V(G_{1})\cup V(G_{2})$ and edge set
$E(G_{1})\cup E(G_{2})$. The \emph{join} $G_1\vee G_2$ is the
graph obtained from $G_1 \cup G_2$ by joining each vertex of
$G_{1}$ to each vertex of $G_{2}$. The complete graph on $n$
vertices and its complement are denoted by $K_{n}$ and $I_{n}$.
Let $X$ and $Y$ be two disjoint subsets of $V(G)$, the
number of edges of $G$ from $X$ to $Y$ is denoted by $e(X,Y)$. For
other terminologies and notations not defined in this paper, we
refer the readers to \cite{BM1}.

A \emph{matching} $M$ of $G$ is a subset of $E(G)$ in which no two
edges have a common end-vertex. $M$ is said to be a \emph{perfect
matching} if it covers all vertices of $G$. A graph $G$ is said to
be \emph{$k$-extendable} for $0\leq k\leq (\nu-2)/2$ if it is
connected, contains a matching of size $k$ and any matching in $G$
of size $k$ is contained in a perfect matching of $G$. $G$ is said
to be \emph{minimal $k$-extendable} if $G$ is $k$-extendable and
$G-e$ is not $k$-extendable for each $e\in E(G)$. The concept of
$k$-extendable graphs was introduced by Plummer in \cite{P1}. In
\cite{Y3}, Yu generalized the idea of $k$-extendibility to
$k\frac{1}{2}$-extendibility for graph of odd order. A graph $G$
is said to be \emph{$k\frac{1}{2}$-extendable} if (1) for any
vertex $v$ of $G$ there exists a matching of size $k$ in $G-v$,
and (2) for every vertex $v$ of $G$, every matching of size $k$ in
$G-v$ is contained in a perfect matching of $G-v$.

A graph $G$ is said to be \emph{$n$-factor-critical}, or
\emph{$n$-critical}, for $0\leq n \leq \nu-2$, if $G-S$ has a
perfect matching for any $S\subseteq V(G)$ with $|S|=n$. For
$n=1$, 2, that is \emph{factor-critical} and \emph{bicritical}.
$G$ is called \emph{minimal $n$-factor-critical} if $G$ is
$n$-factor-critical but $G-e$ is not $n$-factor-critical for any
$e\in E(G)$. The concept of $n$-factor-critical graphs was
introduced by Favaron \cite{F} and Yu \cite{Y3}, independently.

It is easy to verify the following theorem.
\begin{theorem}\label{ob1}
For $0\leq k\leq \nu/2-1$, a $2k$-factor-critical graph is
$k$-extendable, and a $(2k+1)$-factor-critical graph is
$k\frac{1}{2}$-extendable.
\end{theorem}

The reverse of Theorem \ref{ob1} does not hold in general. For
example, a $k$-extendable bipartite graphs can not be
$n$-factor-critical for any $n>0$. However, there has been lots of
research on the relationship between $k$-extendable non-bipartite
graphs and $n$-factor-critical graphs. Most of the results can be
viewed as answers to the following problem, which has been posted
by Favaron \cite{F1} and Yu \cite{Y1}, in slightly different
forms.

\begin{problem}\label{pl1}
Does there exist a non-null function $f(k)$ such that every
$k$-extendable non-bipartite graph of even order $\nu\geq 2k + 2$
is $f(k)$-factor-critical?
\end{problem}

The following two results of Plummer \cite{P1} are answers to
$k=2$, 3.

\begin{theorem}\label{lmp3}
Let $G$ be 2-extendable and non-bipartite with $\nu\geq 6$, then
$G$ is bicritical.
\end{theorem}

\begin{theorem}\label{lmp3}
Let $G$ be 3-extendable and bicritical with $\nu\geq 8$, then
$G-e$ is again bicritical for any $e\in E(G)$.
\end{theorem}

And they have been generalized for all $k\geq 0$, as below.

\begin{theorem}\label{lmf1} $($Favaron \cite{F1}, Liu and Yu \cite{LiuY}$)$
For even integer $k\geq 0$, every connected, non-bipartite,
$k$-extendable graph of even order $\nu>2k$ is
$k$-factor-critical.
\end{theorem}

\begin{theorem}\label{lmf2} $($Favaron \cite{F1}$)$
For even integer $k\geq 0$, every connected non-bipartite,
$(k+1)$-extendable graph $G$ of even order $\nu\geq 2k+4$ is
$k$-factor-critical. Moreover, $G-e$ is $k$-factor-critical for
every edge $e$ of $G$.
\end{theorem}

In light of Theorem \ref{ob1}, if under some conditions
$k$-extendable graphs are $2k$-factor-critical, then the two
classes of graphs are equal. The following results show that this
happens when $k$ is large relative to $\nu$.

\begin{theorem}\label{lmfs1}
$($Favaron and Shi \cite{FS}$)$ Every $((\nu/2)-2)$-extendable
non-bipartite graph with $\nu\geq 14$ is
$(\nu-4)$-factor-critical.
\end{theorem}
\begin{theorem}\label{lmy1}
$($Yu \cite{Y1}$)$ Let $G$ be a non-bipartite graph of even order
and $k$ an integer. If $G$ is $k$-extendable and $k\geq 2(\nu+1)/3$,
then $G$ is $2k$-factor-critical.
\end{theorem}

Following this direction, we give a better lower bound of $k$ and
show that it is the best possible. Furthermore, we show a similar
equivalent relationship between $(2k+1)$-factor-critical graphs
and $k\frac{1}{2}$-extendable graphs.

The following lemmas will be used in the proofs of the main
results.
\begin{lemma}\label{lmp1}
$($Plummer \cite{P1}$)$ If $G$ is a $k$-extendable graph on $\nu\geq
2k+2$ vertices where $k\geq 1$, then $G$ is also $(k-1)$-extendable.
\end{lemma}

\begin{lemma}\label{lmy4}
$($Yu \cite{Y3}$)$ If $G$ is a $k\frac{1}{2}$-extendable graph,
then $G$ is also $(k-1)\frac{1}{2}$-extendable.
\end{lemma}

\begin{lemma}\label{lmly1} $($Lou and Yu \cite{LY}$)$ If $G$
is a $k$-extendable graph with $k\geq \nu/4$, then either $G$ is
bipartite or $\kappa(G)\geq 2k$.
\end{lemma}

\begin{lemma}\label{lmz1}
If $G$ is a $k$-extendable graph, then $G$ is also $m$-extendable
for all integers $0\leq m\leq k$. If $G$ is a
$k\frac{1}{2}$-extendable graph, then $G$ is also
$m\frac{1}{2}$-extendable for all integers $0\leq m\leq k$.
\end{lemma}
\begin{proof} Apply repeatedly Lemma \ref{lmp1} and Lemma \ref{lmy4}.
\end{proof}

\section{Equivalence between extendibility and \\ factor-criticality}
\begin{theorem}\label{th21}
If $k\geq(\nu+2)/4$, then a non-bipartite graph $G$ is
$k$-extendable if and only if it is $2k$-factor-critical .
\end{theorem}

\begin{proof} We only need to prove that if $k\geq(\nu+2)/4$, a $k$-extendable
non-bipartite graph is $2k$-factor-critical.

Let $G$ be a $k$-extendable non-bipartite graph satisfying
$k\geq(\nu+2)/4$ but not $2k$-factor-critical. Then, there exists
a vertex set $S\subseteq V(G)$ with order $2k$, such that $G-S$
has no perfect matching. Moreover, we choose $S$ so that the size
of the maximum matching of $G[S]$ has the maximum value $r_0$.
Clearly, $r_0\leq k-1$.

Let $M_S$ be a maximum matching of $G[S]$, then there exists two
vertices $u_{1}$ and $u_{2}$ in $G[S]$ that are not covered by
$M_S$. By Lemma \ref{lmz1}, $M_S$ is contained in a perfect
matching $M$ of $G$. Let $u_{i}v_{i}\in M$, where $v_{i}\in
V(G-S)$, $i=1$, 2. Let $S^{\prime}=(S\backslash \{u_{2})\}\cup
\{v_{1}\}$. Then $M_S\cup\{u_1v_1\}$ is a matching of
$G[S^\prime]$ of size $r_0+1$. By the choice of $S$,
$G-S^{\prime}$ has a perfect matching $M_{\bar{S^{\prime}}}$, and
$|M_{\bar{S^{\prime}}}|\leq k-1$. By Lemma \ref{lmz1},
$M_{\bar{S^{\prime}}}$ is contained in a perfect matching
$M^\prime$ of $G$. Clearly, $M^\prime\cap E(G[S^\prime])$ is a
perfect matching of $G[S^\prime]$ and $M^\prime\cap E(G[S])$ is a
matching of $G[S]$ of size $k-1$. Therefore, $r_0=k-1$. Then
$M_{\bar{S}}=M\cap E(G-S)$ is a maximum matching of $G-S$ of size
$r=|V(G-S)|/2-1\leq k-2$, and $v_1v_2\notin E(G)$.

Let $M_S=\{x_{1}x_{2},\ldots, x_{2k-3}x_{2k-2}\}$,
$M_{\bar{S}}=\{y_{1}y_{2},\ldots, y_{2r-1}y_{2r}\}$. If $v_{1}y_{1}$,
$v_{2}y_{2}\in E(G)$, then $M_{\bar{S}}\cup \{v_{1}y_{1},
v_{2}y_{2}\}\backslash{y_{1}y_{2}}$ is a perfect matching of $G-S$,
contradicting our assumption. Hence $|\{v_{1}y_{1}, v_{2}y_{2}\}\cap
E(G)| \leq 1$. Similarly, $|\{v_{1}y_{2}, v_{2}y_{1}\}\cap E(G)|
\leq 1$. So $e(\{v_{1}, v_{2}\}, \{y_{1},y_{2}\})\leq 2$. Similarly,
$e(\{v_{1}, v_{2}\}, \{y_{2i-1},y_{2i}\})\leq 2$ for $1\leq i \leq
r$.

If $v_{1}x_{1}, v_{2}x_{2}\in E(G)$, then $M_{\bar{S}}\cup
\{v_{1}x_{1}, v_{2}x_{2}\}$ is a matching of $G$ of size no more
than $k$. By Lemma \ref{lmz1}, $M_{\bar{S}}\cup \{v_{1}x_{1},
v_{2}x_{2}\}$ is contained in a perfect matching
$M^{\prime\prime}$ of $G$. But then $(M^{\prime\prime}\cap
E(G[S]))\cup \{x_{1}x_{2}\}$ is a perfect matching of $G[S]$, a
contradiction. Hence, $|\{v_{1}x_{1}, v_{2}x_{2}\}\cap E(G)|\leq
1$. Similarly, $|\{v_{1}x_{2}, v_{2}x_{1}\}\cap E(G)|\leq 1$. So
$e(\{v_{1}, v_{2}\}, \{x_{1},x_{2}\})\leq 2$. Similarly,
$e(\{v_{1}, v_{2}\}, \{x_{2i-1},x_{2i}\})\leq 2$ for $1\leq i\leq
k-1$.

Then we have
$$d(v_{1})+d(v_{2})
\leq 2(k-1)+2r+4 \leq 2(k-1)+2(k-2)+4 = 4k-2.
$$

But by Lemma \ref{lmly1}, $\delta(G)\geq \kappa(G) \geq 2k$. So
$d(v_{1})+d(v_{2})\geq 2k+2k=4k$, a contradiction.
\end{proof}

To show that the lower bound in Theorem \ref{th21} is best
possible, we consider the following class of graphs. Let $G^{(k)}
=(K_{2k-1}\cup K_{1})\vee (K_{2k-1}\cup K_{1})$, $k\geq 2$. Then
$\nu(G^{(k)})=4k$ and $G^{(k)}$ is non-bipartite.

\begin{theorem}\label{th22}
$G^{(k)}$ is $k$-extendable but not $2k$-factor-critical.
\end{theorem}

\begin{proof} Let $G_{1}$ and $G_{2}$ be two copies of $K_{2k-1}\cup K_{1}$
and $G^{(k)}=G_{1}\vee G_{2}$. $G^{(k)}$ is not
$2k$-factor-critical, since $G_{2}=G^{(k)}-V(G_{1})$ does not have
a perfect matching. Now we prove that $G^{(k)}$ is $k$-extendable.

Let $M$ be a matching of size $k$ in $G$, we show that
$G^{(k)}-V(M)$ has a perfect matching. Let $|M\cap E(G_{i})|=k_{i}$,
$i=1,2$, then $k_{1},k_{2}\leq k-1$. Without lose of generality we
suppose $k_{1}\geq k_{2}$. The size of the maximum matching in
$G_{2}-V(M)$ is no less than
$\lfloor(2k-1-(k-k_{1}-k_{2})-2k_{2})/2\rfloor = \lfloor((k-1)+(k_{1}+k_{2})-2k_{2})/2\rfloor
\geq \lfloor (2k_{1}-2k_{2})/2 \rfloor
= k_{1}-k_{2}.$
Therefore we can find a matching $M^{\prime}$ of size $k_{1}-k_{2}$ in
$G_{2}-V(M)$.

In $G^{(k)}-V(M)-V(M^\prime)$, half of the vertices
are from $G_1$ and the other half are from $G_2$, hence we can
find nonadjacent edges from $G_{1}$ to $G_{2}$ covering all vertices in
it. So we get a perfect matching in $G^{(k)}-V(M)$ and $G^{(k)}$
is $k$-extendable.
\end{proof}

Now we divert our attention to $k\frac{1}{2}$-extendable graphs and
$(2k+1)$-factor-critical graphs. Note that by definition a
$k\frac{1}{2}$-extendable graph can never be bipartite.

\begin{theorem}\label{th23}
If $k\geq(\nu-3)/4$, then a graph $G$ is $k\frac{1}{2}$-extendable
if and only if it is $(2k+1)$-factor-critical.
\end{theorem}

\begin{proof} We only need to prove that for $k\geq(\nu-3)/4$, a
$k\frac{1}{2}$-extendable graph $G$ is $(2k+1)$-factor-critical.

Suppose that $G$ is a $k\frac{1}{2}$-extendable graph with
$k\geq(\nu-3)/4$, but not $(2k+1)$-factor-critical. Then, there
exists a set $S\subseteq V(G)$ of order $2k+1$, such that there is
no perfect matching in $G-S$. Denote by $r$ the size of the
maximum matching in $G[S]$. Clearly, $r\leq k-1$.

Let $M_S$ be a maximum matching of $G[S]$, and $v_{1}$ be a vertex
of $G[S]$ not covered by $M_S$. Then by Lemma \ref{lmz1}, $M_S$ is
contained in a perfect matching $M$ of $G-v_{1}$. Then $M\cap
E(G-S)$ is a matching of $G-S$ of size at most $|V(G-S)|/2-1\leq
k$.

Let $v$ be a vertex in $G-S$ not covered by $M\cap E(G-S)$, then
$M\cap E(G-S)$ is contained in a perfect matching $M^\prime$ of
$G-v$. But $M^\prime\cap E(G[S])$ is a matching of $G[S]$ of size at
least $r+1$, a contradiction.
\end{proof}

We present a class of graphs below to show that the bound in
Theorem \ref{th23} is best possible. Let $H^{(k)}=I_{k+2}\vee
(K_{k+3}\cup K_{2k})$. Then $\nu(H^{(k)})=4k+5$.

\begin{theorem}\label{th24}
$H^{(k)}$ is $k\frac{1}{2}$-extendable but not
$(2k+1)$-factor-critical.
\end{theorem}

\begin{proof} Let $H_{1}=I_{k+2}$, $H_{2}=K_{k+3}$, $H_{3}=K_{2k}$ and $H^{(k)}=H_{1}\vee (H_{2}\cup H_{3})$.

Let $S_{1}$ be a subset of $V(H_{2})$ of order $k-2$ and $u\in
V(H_{3})$. Let $S_0=V(H_{1})\cup S_{1} \cup \{u\}$. Then
$|S_0|=2k+1$ and $H^{(k)}-S_0$ does not have a perfect matching.
Therefore $H^{(k)}$ is not $(2k+1)$-factor-critical.

To prove the $k\frac{1}{2}$-extendibility of $H^{(k)}$, we let
$v\in V(H^{(k)})$, $M$ be a matching of size $k$ in $H^{(k)}-v$
and $S=\{v\}\cup V(M)$. We show that $H^{(k)}-S$ has a perfect
matching.

Let $V_{1}=V(H_{1})-S$, $V_{2}=V(H_{2})-S$ and $V_{3}=V(H_{3})-S$.
The existence of a perfect matching in $H^{(k)}-S$ is equivalent
to the existence of a partition of $V_{1}$ into two subsets
$V_{1}^{\prime}$ and $V_{1}^{\prime\prime}$, such that
$|V_{1}^{\prime}|\leq |V_{2}|$, $|V_{1}^{\prime\prime}|\leq
|V_{3}|$, $|V_{1}^{\prime}|\equiv |V_{2}|$ (mod 2) and
$|V_{1}^{\prime\prime}|\equiv|V_{3}|$ (mod 2). Since
$|V_{1}|+|V_{2}|+|V_{3}|=|V(G)|-(2k+1)$ is even, $|V_{1}|$ and
$|V_{2}|+|V_{3}|$ have the same parity. And since
$|V_{2}|+|V_{3}|\geq (k+3)+2k-(2k+1)=k+2\geq|V_{1}|\geq
k+2-1-k=1$, such a partition can always be obtained. Hence we find
a perfect matching in $H^{(k)}-S$ and $H^{(k)}$ is
$k\frac{1}{2}$-extendable.
\end{proof}

\section{Final remarks}
As we have pointed out earlier, a $k$-extendable bipartite graph
$G$ can not be $n$-factor-critical for any $n>0$. This is because
we can choose a vertex set $S$ of order $n$ so that $G-S$ is not
balanced. However, for $n=2k$, if we keep the two partitions of
$G-S$ balanced when we choose $S$, then $G-S$ does have a perfect
matching. This is a result by Plummer \cite{P2}.

\begin{theorem}\label{lmp2}
Let $G$ be a connected bipartite graph with bipartition (U,W) and
suppose $k$ is a positive integer such that $k\leq \nu/2-1$. Then
$G$ is $k$-extendable if and only if for all $u_{1},\ldots,u_{k}\in
U$ and $w_{1},\ldots,w_{k}\in W$,
$G^{\prime}=G-u_{1}-\cdots-u_{k}-w_{1}-\cdots-w_{k}$ has a perfect
matching.
\end{theorem}

Hence, following the terms in the definition of
$n$-factor-critical graphs, if we define \textquotedblleft
$2k$-factor-criticality\textquotedblright \ in a balanced
bipartite graph $G$ so that we keep the two partitions of $G-S$
balanced when choosing $S$, then $G$ is $k$-extendable if and only
if it is \textquotedblleft$2k$-factor-critical\textquotedblright,
for $0\leq k\leq \nu/2-1$. $\newline\newline\indent$Plummer
\cite{P1} has proved that $\kappa(G) \geq k+1$ for a
$k$-extendable graph $G$. Hence $\delta(G)\geq\kappa(G)\geq k+1$.
For minimal $k$-extendable bipartite graphs, the following result
of Lou \cite{L1} shows that the bound can always be reached.
\begin{theorem}\label{lml1}
Every minimal $k$-extendable bipartite graph $G$ with bipartition
(U,W) has at least $2k+2$ vertices of degree $k+1$. Furthermore,
both $U$ and $W$ contain at least $k+1$ vertices of degree $k+1$.
\end{theorem}
While for minimal $k$-extendable non-bipartite graphs we have not
found such a simple characterization. When $k=1$, the minimum
degree can be 2 or 3. And no result is known for $k\geq 2$.
Illuminated by Lemma \ref{lmly1}, Lou and Yu \cite{LY} raised the
following conjecture.
\begin{conjecture}\label{con1}
Let $G$ be a minimal $k$-extendable graph on $\nu$ vertices with
$\nu/2+1\leq 2k+1$. Then $\delta(G)=k+1$, $2k$ or $2k+1$.
\end{conjecture}
For minimal $n$-factor-critical graphs, Favaron and Shi \cite{FS}
raised the following conjecture.
\begin{conjecture}\label{con2}
Every minimal $n$-factor-critical graph $G$ has $\delta(G)=n+1$.
\end{conjecture}
By the results obtained, we see that except the case that
$\nu=4k$, Conjecture \ref{con1} is actually part of Conjecture
\ref{con2} and the value $2k$ in Conjecture 1 can be excluded.
\section*{Acknowledgements}
We thank Professor Qinglin Yu for suggesting the problem and his
valuable advices.

\end{document}